\documentclass{article}
\usepackage[margin=1.5in]{geometry} 
\usepackage{graphicx} 

\pdfoutput = 1

\usepackage{amsmath,amsthm,amssymb,amsfonts, fancyhdr, comment, parskip, tikz, caption, subcaption, graphicx, xcolor}
\usepackage[colorlinks]{hyperref}
\usetikzlibrary{graphs}
\usetikzlibrary{shadings}
\usetikzlibrary{shapes.geometric}
\usetikzlibrary{arrows}
\usepackage{mathtools}
\usepackage{standalone}
\usepackage{upgreek}
\usepackage{bm}
\usepackage{tikz-cd}
\usepackage{enumerate}
\usepackage{enumitem}
\usepackage{soul}

\hypersetup{
    colorlinks=true,
    linkcolor=blue,
    citecolor=green,
    filecolor=magenta,
    urlcolor=blue,
    pdfborder={0 0 0}, 
}

\usepackage{mdframed}
\usepackage[capitalise]{cleveref}
\usepackage{float}

\newtheorem{thm}{Theorem}

\newtheorem{lemma}[thm]{Lemma}

\newtheorem{claim}[thm]{Claim}

\theoremstyle{definition}

\newcommand\cG{{\mathcal G}}
\newcommand\cH{{\mathcal H}}

\newcommand\cM{{\mathcal M}}

\title{Generalized Ramsey Numbers in the Hypercube}
\author{Emily Heath \footnote {California State Polytechnic University Pomona, \texttt{eheath@cpp.edu}.}
\and Coy Schwieder \footnote {Iowa State University, \texttt{cschwi@iastate.edu}.}
\and Shira Zerbib \footnote {Iowa State University, \texttt{zerbib@iastate.edu}. Supported by NSF CAREER award no. 2336239  and Simons Foundation award no. MP-TSM-00002629.}}
\date{}
\begin{document}
\maketitle

\begin{abstract}


We study the generalized Ramsey numbers $f(Q_n, C_{k}, q)$, that is, the minimum number of colors needed to edge-color the hypercube $Q_n$ so that every copy of the cycle $C_{k}$ has at least $q$ colors. Our main result is that  for any integers $k,q$ satisfying  $k \geq 6$ and $3 \leq q \leq k/2+1$,  we have $f(Q_n, C_{k}, q)= o\left( n^{\frac{k/2-1}{k-q+1}} \right).$ We also prove a few other upper and lower bounds in the special cases $k=4$ and $k=6$. 
 This continues the line of research initiated by Faudree, Gy\'arf\'as, Lesniak, and Schelp  \cite{FGLS} and  Mubayi and Stading \cite{MS} who studied the case $k=q$, and by Conder \cite{conder} who considered the case $k=6$ and  $q=2$. 
\end{abstract}
\section{Introduction}

The \emph{generalized Ramsey number} $f(G, H, q)$, introduced by Erd\H{o}s and Shelah \cite{Erdos-Shelah1} in 1974, is the minimum number of colors needed to edge-color a graph $G$ so that every copy of a subgraph $H$ has at least $q$ colors. Finding $f(G, H, q)$ when $G=K_n$, $H=K_p$, and $q = 2$ is equivalent to determining the diagonal multicolor Ramsey numbers $R_k(p)$ for $k=f(G, H, q)$. 

The problem of bounding $f(K_n, K_p, q)$ was first systematically studied by Erd\H{o}s and Gy\'arf\'as \cite{er-gy}, who proved a general upper bound 
\[f(K_n, K_p, q) = O \left( n^{\frac{p-2}{\binom{p}{2} - q + 1}} \right).\]
Building on work of Bennett, Dudek, and English \cite{BDE},  Bennett, Delcourt, Li, and Postle \cite{BDLP} improved this result by a logarithmic factor except at the integer powers of $n$, where the original upper bound is tight.  More generally, in \cite{BDLP} it was shown that for every graph $G$ on $n$ vertices, subgraph $H$ of $G$, and integer $1\leq q\leq |E(H)|$ such that $|E(H)|-q+1$ does not divide $|V(H)|-2$, 
\begin{equation}\label{eq: general upper bound}
    f(G, H, q) = O \left( \left( \frac{n^{|V(H)| - 2}}{\log n} \right)^{\frac{1}{|E(H)| - q + 1}} \right).
\end{equation} 
The proof utilized the ``conflict-free hypergraph matching method", developed independently by Delcourt and Postle in \cite{DP} and Glock, Joos, Kim, K\"uhn, and Lichev in \cite{GJKKL}. This method  has since been used to study various applications, including  generalized Ramsey numbers,  odd Ramsey numbers, list-colorings, and designs; for example, see \cite{BBHZ, BCD, BHZ, CHHSZ, DP, GJKKL, GHPSZ, JMS, LM}. 

The upper bound in (\ref{eq: general upper bound}) is not known to be tight in general. Many researchers have studied the numbers $f(G, H, q)$ for various host graphs $G$ and subgraphs $H$ (see ~\cite{axenovich2000, AFM, BEHK, BCDP, BDE, 56,  CH1, CH2, CFLS, er-gy, FPS, Mubayi1, mubayi2004, PS, sarkozy2000edge, sarkozy2003application}). In this paper we contribute to this effort by studying the case where the  host graph $G$ is the hypercube and $H$ is a cycle $C_k$. 
Denote by $Q_n$ the 
 {\em $n$-hypercube}, that is, the graph on vertex set $V(Q_n) = \{0, 1\}^n$ (the set of all $\{0,1\}$-sequences of length $n$) where  $E(Q_n)$ consists of all edges $xy, ~x,y\in V(Q_n)$ such that $x$ and $y$ differ in exactly one position.  Note that $Q_n$ is bipartite, hence it has no odd cycles.

The numbers $f(Q_n, C_k, q)$ 
have been studied thus far mostly for the case  $q=k$, namely when all the $k$-cycles of $Q_n$ are required to be rainbow. 
In 1993, Faudree, Gy\'arf\'as, Lesniak, and Schelp  \cite{FGLS} proved that for $n = 4$ or $n \ge 6$, 
\begin{equation}\label{eq: FGLS}
      f(Q_n, C_4, 4) = n. 
\end{equation}
        Mubayi and Stading \cite{MS} later expanded the study of $f(Q_n,C_k,k)$ to all integers $k$ divisible by 4, and to $k=6$. 
\begin{thm}[Mubayi-Stading \cite{MS}]
    For any integer $k\ge 1$  such that $k\equiv 0 \bmod 4$, there exist constants $c_1, c_2$, depending only on $k$, such that
    $$c_1 n^{k/4} \le f(Q_n, C_{k}, k) \le c_2 n^{k/4}.
    $$
In addition, 
    $$
        3n - 2 \leq f(Q_n, C_6, 6) \le n^{1 + o(1)}.
    $$
\end{thm}
In this paper we obtain bounds on $f(Q_n, C_k, q)$ for several parameters $k,q$ such that  $q<k$. One step in this direction was taken by  
 Conder \cite{conder} in 1993, who exhibited  a 3-coloring of $Q_n$ containing no monochromatic $6$-cycles, thus proving
    \[f(Q_n, C_6, 2) \leq 3.\]
    
    Here  we obtain the following bounds.
\begin{thm}\label{thm: generalization}
    For any integers $k,q$ satisfying  $k \geq 3$ and $3 \leq q \leq k+1$,  we have \[ f(Q_n, C_{2k}, q)= o\left( n^{\frac{k-1}{2k-q+1}} \right).\]
\end{thm}
The theorem is proved using the ``bipartite conflict-free matching method" of Delcourt and Postle in \cite{DP}.

In addition, we prove the following lower bounds for 6-cycles.
\begin{thm}\label{thm: 4 color C6}
   We have 
   $$ f(Q_n, C_6, 4) > (n-1)^{1/3},$$ and 
   $$f(Q_n, C_6, 5) > (n-1)^{1/2}.
    $$
   \end{thm}

Note that together with 
 Theorem \ref{thm: generalization},
 this gives  $$ 
        (n-1)^{1/3} \leq f(Q_n, C_6, 4) \leq o(n^{2/3}).
    $$

Finally, we consider the case  $k=4$. Observe that for all $n \geq 2$, 
    $$
        f(Q_n, C_4, 2) = 2.
    $$
Indeed, for $i\in [n]$ let $E_i\subset E(Q_n)$ be the set of edges $xy$ such that $x$ has $i-1$ ones and $y$ has $i$ ones. Observe that any copy $C$ of $C_4$ in $Q_n$ has some $i\in [n-1]$ such that $C\cap E_i \neq \emptyset$ and $C\cap E_{i+1} \neq \emptyset$.
Therefore, to avoid monochromatic cycles of length 4, one can color all the edges in even layers $E_{2i}$ red and all  the edges in odd layers $E_{2i+1}$ blue. 

Thus together with (\ref{eq: FGLS}), the picture would be complete for 4-cycles if $f(Q_n, C_4, 3)$ is determined. Note that trivially we have $f(Q_n, C_4, 3) \ge 3$. We show the following.
\begin{thm}\label{thm: 3 color C4}
    For all $n \geq 2$, we have 
    $$
       f(Q_n, C_4, 3) \leq 4.
    $$
\end{thm}

The paper is organized as follows.
In Section \ref{sect: bipartite matching thm} we describe the ``bipartite conflict-free matching method" due to Delcourt and Postle \cite{DP}, our main tool in this paper. Then in \cref{sect: generalization}, we prove \cref{thm: generalization}. \cref{thm: 4 color C6} is proven in \cref{sect: lower}. Lastly, we prove \cref{thm: 3 color C4} in \cref{sect: 3-color c4}. A couple of short remarks are given in Section 6.

\section{The bipartite conflict-free matching method}\label{sect: bipartite matching thm}

We now state the Bipartite Conflict-Free Matching Method. Given a hypergraph $\cG$ and a vertex $v \in V(\cG)$, the \emph{degree} $\deg_\cG(v)$ of $v$ is the number of edges in $\cG$ containing $v$. If $\cG$ is clear from context, we simply write $\deg(v)$. The \emph{codegree} of two vertices $u,v$ in $V(\cG)$ is the number of edges in $\cG$ containing both $u$ and $v$, which we denote by $\deg_\cG(u,v)$ or $\deg(u,v)$ if the hypergraph is clear from context. 

The maximum degree and minimum degree of $\cG$ are denoted by $\Delta(\cG)$ and $\delta(\cG)$, respectively. We say $\cG$ is \emph{$r$-bounded} if all edges of $\cG$ are of size at most $r$; if all edges of $\cG$ are of size $r$ exactly, we say $\cG$ is \emph{$r$-uniform}. 

For a hypergraph $\cG = (A,B)$, we say $\cG$ is \emph{bipartite} with parts $A$ and $B$ if $V(\cG) = A \cup B$ and every edge of $\cG$ contains exactly one vertex from $A$. A set of edges $\cM$ in $\cG$ is called a \emph{matching} if the intersection of any two edges in $\cM$ is empty. We say a matching $\cM$ of $\cG$ is \emph{$A$-perfect} if every vertex of $A$ is in an edge of the matching. 

We say a hypergraph $\cH$ is a \emph{conflict system} for $\cG$ if $V(\cH) = E(\cG)$ and $E(\cH)$ is a set of matchings of $\cG$ of size at least two. We call a matching $\cM$ of $\cG$ \emph{$\cH$-avoiding} if $\cM$ contains no edges of $\cH$.

For a hypergraph $\cH$, the \emph{$i$-degree} of a vertex $v \in V(\cH)$, which we denote $d_{\cH,i}(v)$, is the number of edges of $\cH$ of size $i$ which contain $v$. The \emph{maximum $i$-degree} of $\cH$, which we denote $\Delta_i(\cH)$, is the maximum of $d_{\cH,i}(v)$ over all $v \in V(\cH)$. The \emph{maximum $(k,\ell)$-codegree of $\cH$} is
\[\Delta_{k,\ell}(\cH) := \max_{S \in \binom{V(\cH)}{\ell}} \left| \left\{ e \in E(\cH) : S \subseteq e, |e| = k \right\} \right|.\]
That is, $\Delta_{k, \ell}(\cH)$ is the maximum number of edges in $\cH$ of size $k$ which contain a particular subset of size $\ell$ vertices. 

Next, we define the \emph{common $2$-degree} of distinct vertices $u,v \in V(\cH)$ as
\[\left| \left\{ w \in V(\cH) : uw, vw \in E(\cH)\right\} \right|.\]
Then the \emph{maximum common $2$-degree} of $\cH$ is the maximum common $2$-degree of $u,v$ taken over all pairs of vertices $u, v \in \cH$, where $u, v$ are vertex-disjoint in $\cG$. That is, recalling $u, v \in E(\cG)$, we have $u \cap v = \emptyset$.

Lastly, if $\cG$ is a hypergraph and $\cH$ is a conflict system of $\cG$, the \emph{$i$-codegree} of a vertex $v \in V(\cG)$ and $e \in E(\cG) = V(\cH)$ with $v \notin e$ is the number of edges of $\cH$ of size $i$ that contain $e$ and an edge incident with $v$. The \emph{maximum $i$-codegree} of $\cG$ with $\cH$ is then the maximum $i$-codegree over all vertices $v \in V(\cG)$ and edges $e \in E(\cG) = V(\cH)$ with $v \notin e$.

\begin{thm}[Delcourt-Postle \cite{DP}]\label{thm: bipartite matching}
    For all integers $r,g \geq 2$ and real $\beta \in (0,1)$, there exists an integer $D_\beta \geq 0$ and real $\alpha > 0$ such that the following holds for all $D \geq D_\beta$:\\
    Let $\cG = (A,B)$ be a bipartite $r$-bounded (multi)-hypergraph with codegrees at most $D^{1-\beta}$ such that every vertex in $A$ has degree at least $(1 + D^{-\alpha}) D$ and every vertex in $B$ has degree at most $D$. Let $\cH$ be a $g$-bounded conflict system of $\cG$ with $\Delta_i(\cH) \leq \alpha D^{i-1} \log{D}$ for all $2 \leq i \leq g$ and $\Delta_{k,\ell}(\cH) \leq D^{k-\ell-\beta}$ for all $2 \leq \ell < k \leq g$. If the maximum 2-codegree of $\cG$ with $\cH$ and the maximum common 2-degree of $\cH$ are both at most $D^{1-\beta}$, then there exists an $\cH$-avoiding $A$-perfect matching of $\cG$ and indeed even a set of $D_A - D^{1-\alpha}$ $(\geq D)$ disjoint $\cH$-avoiding $A$-perfect matchings of $G$.
\end{thm}

Note that it follows from the proof of Theorem \ref{thm: bipartite matching} that $\alpha$ is a constant, depending only on $r$. 

\section{Proof of Theorem \ref{thm: generalization}}\label{sect: generalization}

In this section we prove our main result, Theorem~\ref{thm: generalization}, which follows from the statement below.

\begin{thm} \label{thm: actualthm}
For any $c> 0$,  integers $k,q$ satisfying  $k \geq 3$ and $3 \leq q \leq k+1$, and $n$ sufficiently large, we have
    $$
        f(Q_n, C_{2k}, q) \leq c n^{\frac{k-1}{2k-q+1}} + n^{\delta\frac{k-1}{2k-q+1}}
    $$ for some fixed absolute constant $0< \delta < 1$.
\end{thm}

The proof utilizes the bipartite conflict-free matching method.   Our application of the theorem uses the same auxiliary hypergraph (i.e., bipartite graph) as in \cite{BDLP} and \cite{BCD}.

    First, we prove the following lemma, which will be utilized throughout the proof.

    \begin{lemma}\label{lemma: cycle_count}
        Let $x_1 y_1, \dots, x_\ell y_\ell \in E(Q_n)$ and $k\ge \ell$ be an integer. Then $Q_n$ has at most  $O(n^{k-\ell})$ copies of the cycle $C_{2k}$  containing all $x_1 y_1, \dots, x_\ell y_\ell$.
    \end{lemma}

    \begin{proof}
        We  count the number of ways to construct a copy of the cycle $C_{2k}$ containing $x_1 y_1, \dots, x_\ell y_\ell$. 
        
        First, observe that if $x_1 y_1, \dots, x_\ell y_\ell$ are in no copies of $C_{2k}$ together, we are done. Thus, we may assume $x_1 y_1, \dots, x_\ell y_\ell$ are in some copy of $C_{2k}$ together.
        
        Observe that for any vertex $z$ on such a cycle, $z$ differs from $x_1$ in at most $k$ places. Otherwise, the shortest path between $x_1$ and $z$ has more than $k$ vertices in it, and $x_1, z$ cannot occur in a copy of $C_{2k}$ together. Moreover, if $x_1,y_1,z_3,\dots, z_{2k}$ form a copy of $C_{2k}$, then there exists an index set $I\subset [n]$ with $|I| \le k$ such that all the vertices in $y_1,z_3,\dots, z_{2k}$ differ from $x_1$ in places contained in $I$.
    
        Since the set of edges $x_1 y_1, \dots, x_\ell y_\ell$ covers some $j \geq \ell + 1$ vertices, the vertices collectively differ from $x_1$ in at least $\ell$ indices contained in $I$. Thus there are at most $\binom{n-\ell}{k- \ell } = O \left(n^{k-\ell} \right)$ ways to choose the remaining places which differ from $x_1, y_1, \dots, x_\ell, y_\ell$, and, since $k$ is constant, there are $O(1)$ ways to choose the manner in which the vertices $z_{j+1},\dots, z_{2k}$  differ from $x_1, y_1, \dots, x_\ell, y_\ell$.
    \end{proof}

\begin{proof}[Proof of \cref{thm: actualthm}]
    
    Fix $0 < \varepsilon \le \min\{1,c\}$.
    Let $r=2$, $g=2k$, and $\beta < 1$, and note that $\beta < 1 < \frac{2k-q+2}{k-1}$. Let $n$ be large enough so that $D = \varepsilon n^{\frac{k-1}{2k-q+1}} > \max\{D_\beta, 1\}$, where $D_{\beta}$ is obtained from \cref{thm: bipartite matching}, and so that (\ref{eq: degree}) and (\ref{eq: codegree}) hold below.
    Choose $\delta$  so that  $1 - \alpha < \delta < 1$, where $\alpha$ is a fixed constant obtained from Theorem \ref{thm: bipartite matching}. 

    Let $N = \left[ \varepsilon n^{\frac{k-1}{2k-q+1}} + n^{\delta \frac{k-1}{2k-q+1}} \right]$,
    where by $[a]$ we denote the set $\{1,2,\dots,\lceil a \rceil\}$. The set $N$ will be the color palette for the edge coloring of $Q_n$.
   
    We now define the auxiliary hypergraph $\cG$ (in our case it will be a graph). Let $A = E(Q_n)$. 
    For each $i \in N$, let $B_i$ be a copy of $E(Q_n)$, where we denote any edge $e \in E(Q_n)$ as $e_i$ in $B_i$. Take $B = \bigcup_{i \in N} B_i$. Let $\cG $ be the bipartite graph with vertex sides $A$ and $B$ and edges 
    $$
        e_{xy, i} = \{ xy, xy_i \}
    $$
    for all $xy \in E(Q_n)$ and $i \in N$. We refer to the edges of $\cG$ as \emph{tiles}. 
    Observe that $\cG$ is $2$-uniform, and thus $r$-bounded for $r=2$. 
    Observe that an $A$-perfect matching in $\cG$ corresponds to a well-defined edge-coloring of all the edges of $Q_n$ with at most $\varepsilon n^{\frac{k-1}{2k-q+1}} + n^{\delta \frac{k-1}{2k-q+1}}$ colors. 
    
    We now check that $\cG$ has degrees and codegrees which satisfy the requirements of \cref{thm: bipartite matching}. 

    \begin{claim}
        Our auxiliary graph $\cG$ has codegrees at most $D^{1 - \beta}$. Additionally, for any $a \in A$, we have $\deg(a) \geq (1 + D^{-\alpha})D$. Lastly, for any $b \in B$, we have $\deg(b) \leq D$.
    \end{claim}

    \begin{proof}
        First, note that $\cG$ is a graph, so the codegree of any two vertices is at most 1, which is smaller than $D^{1-\beta}$. 
        We show that $\deg(a)\ge  D+D^\delta > (1+D^{-\alpha})D $ for every $a\in A$.
        Indeed,
        $$
            \deg(a) = |N| 
            = \varepsilon n^{\frac{k-1}{2k-q+1}} + n^{\delta \frac{k-1}{2k-q+1}}\ge   \varepsilon n^{\frac{k-1}{2k-q+1}} + \varepsilon^{\delta} n^{\delta \frac{k-1}{2k-q+1}} =  D+D^\delta.
        $$
        Finally, for a vertex $b\in B$ we have 
        $
            \deg(b) = 1 \le D.
        $
    \end{proof}

Next, we define a conflict system $\cH$ as follows. Let $V(\cH) = E(\cG)$. For every $0 \leq j \leq q - 2$, we define $E(\cH)$ to contain all edges of the form 
    $$
        H= \left\{ e_{x_1y_1,i_1},e_{x_2y_2,i_2},\ldots, e_{x_{2k-j}y_{2k-j},i_{2k-j}}\right\},
    $$
such that the following hold: 
\begin{itemize}
    \item $x_1y_1, \dots, x_{2k-j}y_{2k-j}$ are distinct edges all belonging to the same copy of $C_{2k}$ in $Q_n$, 
    \item the number of distinct colors in $\{i_1,\dots,i_{2k-j}\}$ is at most $q-1-j$, and
\item every color in the multiset $\{i_1,\dots,i_{2k-j}\}$ appears at least twice.
\end{itemize}

We will refer to the edges of $\cH$ as \emph{conflicts}.

Note that since $2k-j \geq 2k-(q-2)\geq k+1$, the edges of $Q_n$ in this conflict form more than half the edges of the corresponding copy of $C_{2k}$. Finally, observe that $\cH$ is a $2k$-bounded conflict system of $\cG$. 

See Figure \ref{fig: C6_3_colors_conflicts} for examples of colored subgraphs of $Q_n$ which correspond to conflicts of $\cH$ in the case $k = 3, q = 3$.

    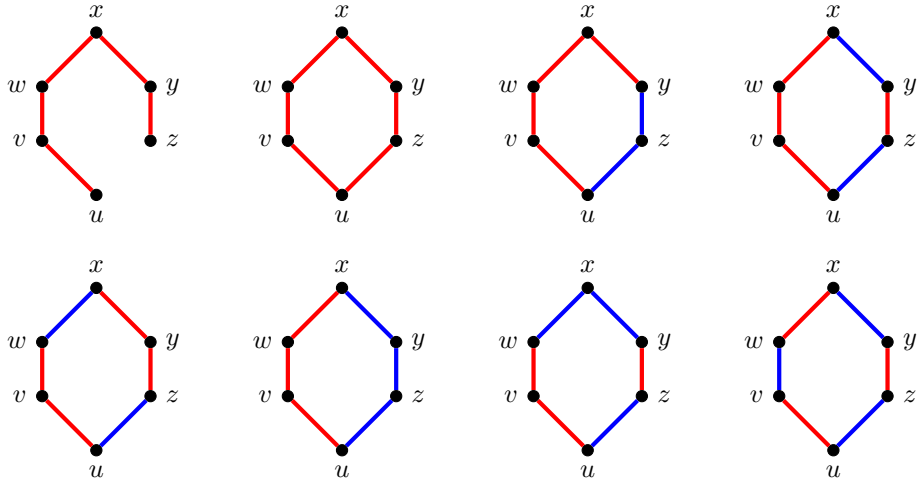
\begin{figure}[h!]
        \centering
            \begin{tikzpicture}[scale=.9]
            \node[circle, fill, draw, minimum size=1pt, inner sep = 1.5pt, label=below:{$u$}] (u) at (0,0) {};
            \node[circle, fill, draw, minimum size=1pt, inner sep = 1.5pt, label=left:{$v$}] (v) at (-.8,.8) {};
            \node[circle, fill, draw, minimum size=1pt, inner sep = 1.5pt, label=left:{$w$}] (w) at (-.8,1.6) {};
            \node[circle, fill, draw, minimum size=1pt, inner sep = 1.5pt, label=above:{$x$}] (x) at (0,2.4) {};
            \node[circle, fill, draw, minimum size=1pt, inner sep = 1.5pt, label=right:{$y$}] (y) at (.8,1.6) {};
            \node[circle, fill, draw, minimum size=1pt, inner sep = 1.5pt, label=right:{$z$}] (z) at (.8,.8) {};
            
            \draw[ultra thick, color = red] (u) -- (v) -- (w) -- (x) -- (y) -- (z);
        
    \end{tikzpicture}
        \hspace{.5cm}
            \begin{tikzpicture}[scale=.9]
            \node[circle, fill, draw, minimum size=1pt, inner sep = 1.5pt, label=below:{$u$}] (u) at (0,0) {};
            \node[circle, fill, draw, minimum size=1pt, inner sep = 1.5pt, label=left:{$v$}] (v) at (-.8,.8) {};
            \node[circle, fill, draw, minimum size=1pt, inner sep = 1.5pt, label=left:{$w$}] (w) at (-.8,1.6) {};
            \node[circle, fill, draw, minimum size=1pt, inner sep = 1.5pt, label=above:{$x$}] (x) at (0,2.4) {};
            \node[circle, fill, draw, minimum size=1pt, inner sep = 1.5pt, label=right:{$y$}] (y) at (.8,1.6) {};
            \node[circle, fill, draw, minimum size=1pt, inner sep = 1.5pt, label=right:{$z$}] (z) at (.8,.8) {};
            
            \draw[ultra thick, color = red] (u) -- (v);
            \draw[ultra thick, color = red] (w) -- (x);
            \draw[ultra thick, color = red] (y) -- (z);
            \draw[ultra thick, color = red] (v) -- (w);
            \draw[ultra thick, color = red] (x) -- (y);
            \draw[ultra thick, color = red] (u) -- (z);
        
    \end{tikzpicture}
        \hspace{.5cm}
            \begin{tikzpicture}[scale=.9]
            \node[circle, fill, draw, minimum size=1pt, inner sep = 1.5pt, label=below:{$u$}] (u) at (0,0) {};
            \node[circle, fill, draw, minimum size=1pt, inner sep = 1.5pt, label=left:{$v$}] (v) at (-.8,.8) {};
            \node[circle, fill, draw, minimum size=1pt, inner sep = 1.5pt, label=left:{$w$}] (w) at (-.8,1.6) {};
            \node[circle, fill, draw, minimum size=1pt, inner sep = 1.5pt, label=above:{$x$}] (x) at (0,2.4) {};
            \node[circle, fill, draw, minimum size=1pt, inner sep = 1.5pt, label=right:{$y$}] (y) at (.8,1.6) {};
            \node[circle, fill, draw, minimum size=1pt, inner sep = 1.5pt, label=right:{$z$}] (z) at (.8,.8) {};
            
            \draw[ultra thick, color = red] (u) -- (v) -- (w) -- (x) -- (y);
            \draw[ultra thick, color = blue] (y) -- (z) -- (u);
        
    \end{tikzpicture}
        \hspace{.5cm}
            \begin{tikzpicture}[scale=.9]
            \node[circle, fill, draw, minimum size=1pt, inner sep = 1.5pt, label=below:{$u$}] (u) at (0,0) {};
            \node[circle, fill, draw, minimum size=1pt, inner sep = 1.5pt, label=left:{$v$}] (v) at (-.8,.8) {};
            \node[circle, fill, draw, minimum size=1pt, inner sep = 1.5pt, label=left:{$w$}] (w) at (-.8,1.6) {};
            \node[circle, fill, draw, minimum size=1pt, inner sep = 1.5pt, label=above:{$x$}] (x) at (0,2.4) {};
            \node[circle, fill, draw, minimum size=1pt, inner sep = 1.5pt, label=right:{$y$}] (y) at (.8,1.6) {};
            \node[circle, fill, draw, minimum size=1pt, inner sep = 1.5pt, label=right:{$z$}] (z) at (.8,.8) {};
            
            \draw[ultra thick, color = red] (u) -- (v) -- (w) -- (x);
            \draw[ultra thick, color = red] (y) -- (z);
            \draw[ultra thick, color = blue] (x) -- (y);
            \draw[ultra thick, color = blue] (z) -- (u);
        
    \end{tikzpicture}
        
        \vspace{.1in}
            \begin{tikzpicture}[scale=.9]
            \node[circle, fill, draw, minimum size=1pt, inner sep = 1.5pt, label=below:{$u$}] (u) at (0,0) {};
            \node[circle, fill, draw, minimum size=1pt, inner sep = 1.5pt, label=left:{$v$}] (v) at (-.8,.8) {};
            \node[circle, fill, draw, minimum size=1pt, inner sep = 1.5pt, label=left:{$w$}] (w) at (-.8,1.6) {};
            \node[circle, fill, draw, minimum size=1pt, inner sep = 1.5pt, label=above:{$x$}] (x) at (0,2.4) {};
            \node[circle, fill, draw, minimum size=1pt, inner sep = 1.5pt, label=right:{$y$}] (y) at (.8,1.6) {};
            \node[circle, fill, draw, minimum size=1pt, inner sep = 1.5pt, label=right:{$z$}] (z) at (.8,.8) {};
            
            \draw[ultra thick, color = red] (u) -- (v) -- (w);
            \draw[ultra thick, color = red] (x) -- (y) -- (z);
            \draw[ultra thick, color = blue] (w) -- (x);
            \draw[ultra thick, color = blue] (z) -- (u);
        
    \end{tikzpicture}
        \hspace{.5cm}
            \begin{tikzpicture}[scale=.9]
            \node[circle, fill, draw, minimum size=1pt, inner sep = 1.5pt, label=below:{$u$}] (u) at (0,0) {};
            \node[circle, fill, draw, minimum size=1pt, inner sep = 1.5pt, label=left:{$v$}] (v) at (-.8,.8) {};
            \node[circle, fill, draw, minimum size=1pt, inner sep = 1.5pt, label=left:{$w$}] (w) at (-.8,1.6) {};
            \node[circle, fill, draw, minimum size=1pt, inner sep = 1.5pt, label=above:{$x$}] (x) at (0,2.4) {};
            \node[circle, fill, draw, minimum size=1pt, inner sep = 1.5pt, label=right:{$y$}] (y) at (.8,1.6) {};
            \node[circle, fill, draw, minimum size=1pt, inner sep = 1.5pt, label=right:{$z$}] (z) at (.8,.8) {};
            
            \draw[ultra thick, color = red] (u) -- (v) -- (w) -- (x);
            \draw[ultra thick, color = blue] (x) -- (y) -- (z) -- (u);
        
    \end{tikzpicture}
        \hspace{.5cm}
            \begin{tikzpicture}[scale=.9]
            \node[circle, fill, draw, minimum size=1pt, inner sep = 1.5pt, label=below:{$u$}] (u) at (0,0) {};
            \node[circle, fill, draw, minimum size=1pt, inner sep = 1.5pt, label=left:{$v$}] (v) at (-.8,.8) {};
            \node[circle, fill, draw, minimum size=1pt, inner sep = 1.5pt, label=left:{$w$}] (w) at (-.8,1.6) {};
            \node[circle, fill, draw, minimum size=1pt, inner sep = 1.5pt, label=above:{$x$}] (x) at (0,2.4) {};
            \node[circle, fill, draw, minimum size=1pt, inner sep = 1.5pt, label=right:{$y$}] (y) at (.8,1.6) {};
            \node[circle, fill, draw, minimum size=1pt, inner sep = 1.5pt, label=right:{$z$}] (z) at (.8,.8) {};
            
            \draw[ultra thick, color = red] (u) -- (v) -- (w);
            \draw[ultra thick, color = red] (y) -- (z);
            \draw[ultra thick, color = blue] (w) -- (x) -- (y);
            \draw[ultra thick, color = blue] (z) -- (u);
        
    \end{tikzpicture}
        \hspace{.5cm}
            \begin{tikzpicture}[scale=.9]
            \node[circle, fill, draw, minimum size=1pt, inner sep = 1.5pt, label=below:{$u$}] (u) at (0,0) {};
            \node[circle, fill, draw, minimum size=1pt, inner sep = 1.5pt, label=left:{$v$}] (v) at (-.8,.8) {};
            \node[circle, fill, draw, minimum size=1pt, inner sep = 1.5pt, label=left:{$w$}] (w) at (-.8,1.6) {};
            \node[circle, fill, draw, minimum size=1pt, inner sep = 1.5pt, label=above:{$x$}] (x) at (0,2.4) {};
            \node[circle, fill, draw, minimum size=1pt, inner sep = 1.5pt, label=right:{$y$}] (y) at (.8,1.6) {};
            \node[circle, fill, draw, minimum size=1pt, inner sep = 1.5pt, label=right:{$z$}] (z) at (.8,.8) {};
            
            \draw[ultra thick, color = red] (u) -- (v);
            \draw[ultra thick, color = red] (w) -- (x);
            \draw[ultra thick, color = red] (y) -- (z);
            \draw[ultra thick, color = blue] (v) -- (w);
            \draw[ultra thick, color = blue] (x) -- (y);
            \draw[ultra thick, color = blue] (u) -- (z);
        
    \end{tikzpicture}
        \caption{The colored subgraphs of $Q_n$ corresponding to conflicts in $\cH$ when $k = q = 3$.}
        \label{fig: C6_3_colors_conflicts}
    \end{figure}

    \subsection{Degree conditions of  $\cH$}\label{sect: degree}
        We now check that $\cH$ satisfies the  degree condition $\Delta_{2k-j}(\cH) \leq \alpha D^{2k-j-1} \log{D}$.  Let $t = \frac{k-1}{2k - q + 1}$, so the size of our color palette $N$ is $O(n^t)$. 

        \begin{claim}
            For all $0 \leq j \leq q - 2$, we have 
            $$
                \Delta_{2k-j} \left( \cH \right) \leq O \left(n^{t(2k-j-1)} \right).
            $$
        \end{claim}

        \begin{proof}
            Fix a tile $e_{xy,i} \in V(\cH)$. We  count the number of conflicts of size $2k - j$ containing $e_{xy,i}$. We do so by counting the number of ways to construct a subgraph of a $2k$-cycle in $Q_n$ with $2k-j$ distinct edges including $xy$ and $q-1-j$ distinct colors including $i$. 
        
            By \cref{lemma: cycle_count}, the edge $xy$ is in $O\left(n^{k-1} \right)$ $2k$-cycles. For each such copy of $C_{2k}$, there are 
            $${|N| \choose q- j -2}+{|N|\choose q-j-3} +\cdots+{|N|\choose 1}\le   O\left(n^ {t(q - 2 - j)} \right)$$ ways to choose the remaining at most $q-2-j$ colors. Lastly, since $k$ is a constant, there are $O(1)$ ways to select the remaining $2k-j-1$ edges from the copy of $C_{2k}$ and distribute the chosen colors among those edges. Thus, \[\Delta_{2k-j} (\cH) 
                \leq O\left(n^{k-1 + t(q - 2 - j)} \right) = O\left(n^{t( 2k - j - 1)} \right).\]
        \end{proof}
        Since $D = O\left(n^t\right)$, we have
        \begin{equation}\label{eq: degree}
                \Delta_{2k-j}(\cH) \leq  O\left(n^{t( 2k - j - 1)} \right) 
            < \alpha D^{2k-j-1} \log{D}
        \end{equation}
    for large enough $n$.
    
    \subsection{Codegree conditions}\label{sect: codegree}
        We wish to verify that $\cH$ satisfies the  codegree condition $\Delta_{2k-j, \ell} (\cH) \leq D^{2k - j - \ell - \beta}$ for all pairs $(j,\ell)$ such that $2 \leq \ell < 2k - j \leq 2k$.

        \begin{claim}
            For all pairs $(j,\ell)$ such that $2 \leq \ell < 2k - j \leq 2k$, we have
            $$
                \Delta_{2k-j, \ell} \left( \cH \right) < O \left( n^{t(2k - j - \ell -  \beta)} \right).
            $$
        \end{claim}

        \begin{proof}
            Fix $\ell$ tiles. Note that if these tiles do not correspond to $\ell$ distinct edges of $Q_n$ which all appear in at least one $2k$-cycle together, then there are 0 conflicts in $\cH$ containing these $\ell$ tiles. Otherwise, by \cref{lemma: cycle_count}, there are at most $O\left(n^{k-\ell}\right)$ $2k$-cycles in $Q_n$ containing all the corresponding $\ell$ edges. 

            We now break into two cases.
            
            First, assume $2 \leq \ell \leq 2k - j - 2$. Then for any fixed $2k$-cycle, there are at most $O\left( n^{t(q - j - 2)}\right)$ ways to choose the remaining colors in the conflict, and $O(1)$ ways to choose the edges of the cycle to include in the conflict and to assign them colors. Since  
         $$
                \frac{k - \ell}{2k - q - \ell + 2 - \beta}
                = \frac{(k - 1) - \ell + 1}{(2k - q + 1) - \ell +  (1 - \beta)}
                < t,
            $$ we have 
            $$
                \Delta_{2k-j, \ell} (\cH) \leq O\left( n^{k - \ell + t(q - j - 2)} \right) <  O \left( n^{t(2k - j - \ell - \beta)} \right).
            $$
             The second case is  $\ell = 2k - j - 1$. In this case, since all conflicts in $\cH$ of size $2k - j$ have the property that any $2k - j - 1$ tiles in the conflict have at least $q - 1 - j$ distinct colors, we need not choose any more colors for a conflict containing the $\ell$ fixed tiles. Since 
            \[ k - (2k - j - 1)  
                = j - k + 1
                < t(1 - \beta), \] for all $j$, we have
            \[\Delta_{2k-j, 2k-j-1} \left( \cH \right) 
                \leq O\left( n^{k - \ell} \right)
                < O\left(n^{t(2k - j - (2k - j - 1) - \beta)} \right)
                = O \left( n^{t(1 - \beta)} \right).\]
        \end{proof}

By this claim,  for large enough $n$, we have 
    \begin{equation}\label{eq: codegree}
                \Delta_{2k-j, \ell}(\cH) <  \varepsilon^{2k-j-\ell - \beta}n^{t( 2k - j - \ell - \beta)} 
            = D^{2k-j-\ell - \beta}.
        \end{equation}

    Lastly, we check the maximum $2$-codegree of $\cG$ with $\cH$ and the maximum common $2$-degree of $\cH$. 
    \begin{claim}
        The maximum $2$-codegree of $\cG$ with $\cH$ and the maximum common $2$-degree of $\cH$ are both at most $D^{1 - \beta}$.
    \end{claim}
    
    \begin{proof}
        Since there are no edges of size 2 in $\cH$, the maximum $2$-codegree of $\cG$ with $\cH$ and the maximum common $2$-degree of $\cH$ are both trivially 0, which is clearly less than $D^{1 - \beta}$. 
    \end{proof}

    Since all conditions hold, by \cref{thm: bipartite matching}, there exists an $\cH$-avoiding $A$-perfect matching $\cM$ of $\cG$. Note that $\cM$ corresponds to a well-defined coloring of $Q_n$ using at most $|N| = \varepsilon n^{\frac{k-1}{2k-1+1}} + n^{\delta \frac{k-1}{2k-1+1}}$ colors in which every copy of $C_{2k}$ has the property that any $(2k-j)$-edge subgraph $H$ has at least $q - 1 - j$ colors for every $0 
    \leq j \leq q-2$.

    \begin{claim}
        The coloring afforded by \cref{thm: bipartite matching} corresponds to a coloring of $Q_n$ in which no copy of $C_{2k}$ has fewer than $q$ colors.
    \end{claim}

    \begin{proof}
        For sake of contradiction, suppose we have some copy $C$ of $C_{2k}$ with at most $q-1$  colors. Let $j$ be the number of colors that appear in $C$ exactly once. Remove the edges of $C$ that are colored by a color that appears exactly once in $C$. We are left with a subgraph $F$ of $C$ containing at most $q-1-j$ colors and every color appears at least twice. Note also that since $q-1 \le k$ we must have that $0\le j\le q-2$ because at least one of the colors has to appear at least twice.

        We claim that $F$ corresponds to a conflict $H$ of $\cH$, which constitutes a contradiction. Indeed, $F$ contains some $2k-j$ distinct edges $x_1y_1,\dots,x_{2k-j}y_{2k-j}$, all appearing in the same copy $C$ of $C_{2k}$. Moreover, every color on the edges of $F$ appears at least twice, and the number of distinct colors is at most $q-1-j$. 
    \end{proof}

 This completes the proof of Theorem \ref{thm: actualthm}.  
\end{proof}
\section{Proof of Theorem \ref{thm: 4 color C6}}\label{sect: lower}

    To prove $f(Q_n,C_6,4) > (n-1)^{1/3}$, recall that $Q_n$ is isomorphic to $Q_{n-1} \times K_2$.  That is, any vertex $u$ in $Q_n$ can be uniquely written as $0v$ or $1v$ for some $v \in Q_{n-1}$. Let $Q_0$ be the subgraph of $Q_n$ induced on the set of vertices $\{0v \mid v\in Q_{n-1}\}$,  and similarly, let   $Q_1$ be the subgraph of $Q_n$ induced on the set of vertices $\{1v \mid v\in Q_{n-1}\}$. Note that $Q_0$ and $Q_1$ are isomorphic to $Q_{n-1}$, and the edges in $Q_n$ between $Q_0$ and $Q_1$ form a matching containing all edges of the form $\{0v,1v\}$.

    Now suppose we edge-color $Q_n$ with $c$ colors such that every copy of $C_6$ has at least 4 colors. Fix vertex $0v \in V(Q_0)$, and consider its $n-1$ neighbors in $Q_0$.  By the pigeonhole principle, there are at least $t := \frac{n-1}{c}$ edges to the neighbors $0v_1, \dots, 0v_t$ colored by the same color, say color red. Now, consider the set of edges $\{1v,1v_1\}, \dots, \{1v,1v_t\}$ in $Q_1$. Again,  by the pigeonhole principle, there are at least $\frac{t}{c} = \frac{n-1}{c^2}$ of them colored by the same color, say color blue. Without loss of generality, the edges  $\{1v,1v_1\}, \dots, \{1v,1v_{t/c}\}$ are colored blue (see Figure \ref{fig: C6_4_colors_lower}).

    Now observe that the edges $\{0v_1,1v_1\}, \dots,\{0v_{t/c},1v_{t/c}\}$ are all colored by distinct colors, for otherwise we get a copy of $C_6$ with at most 3 colors. This shows that 
    \[c > \frac{t}{c} = \frac{n-1}{c^2},\] implying the desired bound.

    \begin{figure}[h!]
        \centering
        \includegraphics[scale = 0.8]{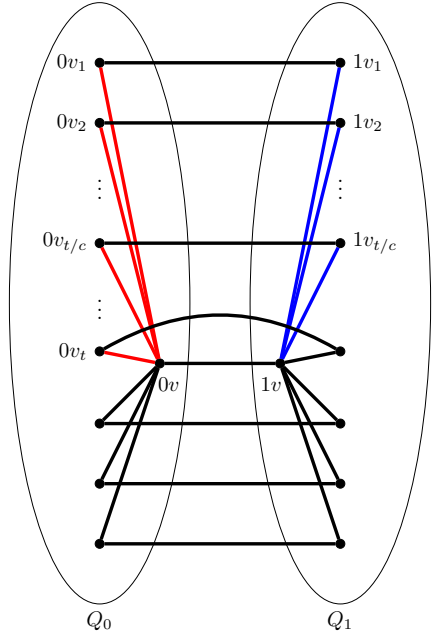}
        \caption{Color classes of $Q_n$.}
        \label{fig: C6_4_colors_lower}
    \end{figure}

\medskip

Now, to see that 
 $f(Q_n, C_6, 5) > (n-1)^{1/2}$, observe that all the edges  $\{0v_1,1v_1\},\dots, \{0v_{n/c},1v_{n/c}\}$ in the previous argument, should now have distinct colors, which are all distinct from red. This shows 
 $c> \frac{n-1}{c}$, as needed. \qed

\section{Proof of Theorem \ref{thm: 3 color C4}}\label{sect: 3-color c4}

   We provide an edge-coloring of $Q_n$ using 4 colors in which any copy of $C_4$ has at least three colors. Versions of this coloring have been discussed and used in \cite{MA} and \cite{MS}.
    
    For a sequence $v$ of 0's and 1's, let $n(v)$ be the number of 1's in $v$. If  $uv$ is an edge in $Q_n$, let $w(uv)\in \{0,1\}$ denote the number of 1's in $u$ (or $v$) before the place of difference between $u$ and $v$. 
    
    For an edge $ab$ where $n(a) < n(b)$, consider the coloring 
    $$
        \phi(ab) = (\phi_1(ab),\phi_2(ab))  = \left( n(a) \bmod{2}, w(ab) \bmod{2} \right).
    $$
   
  Consider a copy $abcd$ of $C_4$ in $Q_n$, as shown in Figure \ref{fig:C4_3_colors}. Without loss of generality, we have $a = x0y0z, b = x0y1z, c = x1y1z$, and $d = x1y0z$, where $x,y,z$ are some fixed sequences of 0's and 1's.
  
    Observe that $\phi_1(ab) +1= \phi_1(ad)+1 = \phi_1(bc) = \phi_1(cd)$.  Further, since $w(ad)=w(bc)$ and 
     $w(cd) = w(ab)+1$, we must have that either $\phi_2(ab) \neq \phi_2(ad)$ or $\phi_2(cd) \neq \phi_2(bc)$, showing that at least 3 of the pairs $\phi(ab), \phi(bc), \phi(cd), \phi(ad)$ are distinct.  
\qed

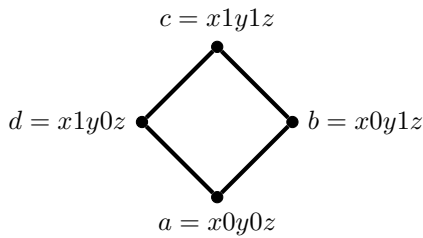
\begin{figure}[h!]
    \centering
        \begin{tikzpicture}
            \node[circle, fill, draw, minimum size=1pt, inner sep = 1.5pt, label=below:{$a = x0y0z$}] (a) at (0,0) {};
            \node[circle, fill, draw, minimum size=1pt, inner sep = 1.5pt, label=right:{$b = x0y1z$}] (b) at (1,1) {};
            \node[circle, fill, draw, minimum size=1pt, inner sep = 1.5pt, label=left:{$d = x1y0z$}] (d) at (-1,1) {};
            \node[circle, fill, draw, minimum size=1pt, inner sep = 1.5pt, label=above:{$c = x1y1z$}] (c) at (0,2) {};
            \draw[ultra thick] (a) -- (b) -- (c) -- (d) -- (a);
        
    \end{tikzpicture}
    \caption{A copy of $C_4$ in $Q_n$.}
    \label{fig:C4_3_colors}
\end{figure}

\section{Concluding remarks}
It is probably possible to generalize Theorem \ref{thm: generalization} by replacing $C_{2k}$ with any subgraph $H$ of $Q_n$ with at most $2k$ edges and diameter bounded by $k$. However, in light of previous work and the proven lower bounds, we preferred to state our results in terms of cycles. 

Furthermore, by optimizing the proof, it may be possible to shave some log power from the bound. However we were mostly interested in the order of magnitude of the exponent. 

\section{Acknowledgment}
We are grateful to Patrick Bennett for  valuable comments and suggestions. 

\bibliography{bibfile}
\bibliographystyle{abbrv}

\end{document}